\documentclass[a4wide]{amsart}
\usepackage{stix}
\usepackage{amsfonts}
\usepackage{amsmath}
\usepackage{amsthm}
\usepackage{mathrsfs}
\usepackage{enumerate}
\usepackage{graphics}
\usepackage{mathtools}
\usepackage{microtype}
\usepackage[defaultlines=4]{nowidow}
\usepackage{xcolor}
\usepackage[english]{babel}
\usepackage{esvect}
\usepackage{complexity}
\usepackage[many]{tcolorbox}
\usepackage{thm-restate}
\usepackage[all]{xy}
\usepackage{cleveref}
\usepackage{relsize}
\usepackage{mathtools}
\usepackage{todonotes}
\usepackage{url}
\usepackage{subfig}
\theoremstyle{remark}

\newcommand{\cqed}{\ensuremath{\lhd}}

\makeatletter
\tcolorboxenvironment{ext_theorem}{
	colframe=lightgray,interior hidden, breakable,before skip=10pt,after skip=10pt,	outer arc=0pt,
	arc=0pt,
	colback=white,
	rightrule=0pt,
	toprule=0pt,
	top=0pt,
	right=0pt,
	bottom=0pt,
	bottomrule=0pt,
}
\tcolorboxenvironment{theorem}{
	colframe=cyan,interior hidden, breakable,before skip=10pt,after skip=10pt,	outer arc=0pt,
	arc=0pt,
	colback=white,
	rightrule=0pt,
	toprule=0pt,
	top=0pt,
	right=0pt,
	bottom=0pt,
	bottomrule=0pt,
}
\tcolorboxenvironment{definition}{
	colframe=olive,interior hidden, breakable,before skip=10pt,after skip=10pt,	outer arc=0pt,
	arc=0pt,
	colback=white,
	rightrule=0pt,
	toprule=0pt,
	top=0pt,
	right=0pt,
	bottom=0pt,
	bottomrule=0pt,
}
\tcolorboxenvironment{fact}{
	colframe=blue, interior hidden, breakable,before skip=10pt,after skip=10pt,	outer arc=0pt,
	arc=0pt,
	colback=white,
	rightrule=0pt,
	leftrule=0.5pt,
	toprule=0pt,
	top=0pt,
	right=0pt,
	bottom=0pt,
	bottomrule=0pt,
}
\tcolorboxenvironment{lemma}{
	colframe=blue,interior hidden, breakable,before skip=10pt,after skip=10pt,	outer arc=0pt,
	arc=0pt,
	colback=white,
	leftrule=0.5pt,
	rightrule=0pt,
	toprule=0pt,
	top=0pt,
	right=0pt,
	bottom=0pt,
	bottomrule=0pt,
}
\tcolorboxenvironment{corollary}{
	colframe=blue,interior hidden, breakable,before skip=10pt,after skip=10pt,	outer arc=0pt,
	arc=0pt,
	colback=white,
	rightrule=0pt,
	toprule=0pt,
	top=0pt,
	right=0pt,
	bottom=0pt,
	bottomrule=0pt,
}

\makeatother

\newcounter{dummyc}
\crefname{fact}{Fact}{Facts}
\crefname{conjecture}{Conjecture}{Conjectures}

\newcommand{\ERCagreement}{This paper is part of projects that have received funding from the European Research Council (ERC) under the European Union's Horizon 2020 research and innovation programme (grant agreements No 810115 -- {\sc Dynasnet}.\\
	\includegraphics[width=.25\textwidth]{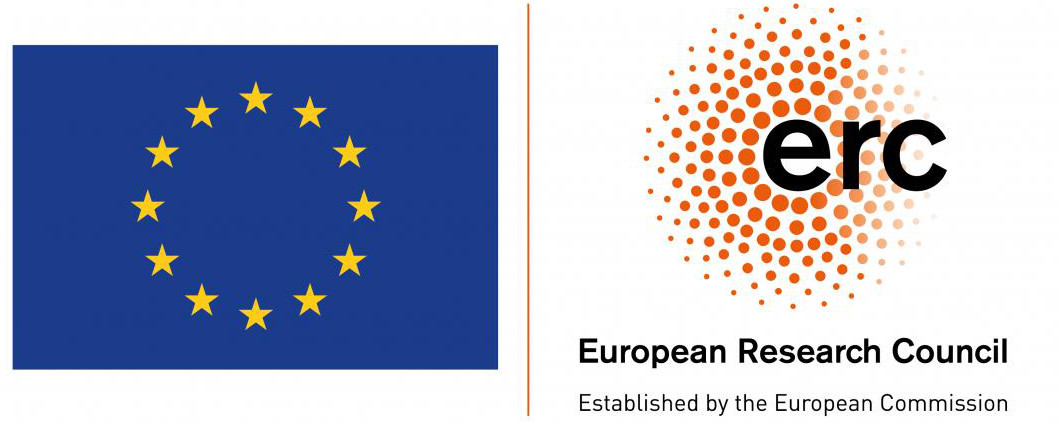}}

\newtheorem{theorem}{Theorem}[section]
\newtheorem{ext_theorem}[theorem]{Theorem}
\newtheorem{corollary}{Corollary}[section]

\newtheorem{lemma}{Lemma}[section]

\newtheorem{problem}{Problem}

\newtheorem{example}{Example}
\newtheorem{fact}[theorem]{Fact}
\crefname{figure}{Figure}{Figures}
\crefname{theorem}{Theorem}{Theorems}
\crefname{definition}{Definition}{Definitions}
\crefname{ext_theorem}{Theorem}{Theorems}
\crefname{corollary}{Corollary}{Corollaries}
\crefname{lemma}{Lemma}{Lemmas}
\crefname{section}{Section}{Sections}

\DeclareMathOperator{\svm}{svm}
\newcommand{\GF}[1]{\ensuremath{\mathbb F}_{#1}}
\DeclareMathOperator{\dist}{\text{dist}}

\begin{document}

	\title{Shallow vertex minors, stability, and dependence}
	\thanks{\ERCagreement}
	\author{Hector Buffi\`ere}\address{\'Ecole Normale Sup\'erieure, Paris, France}\email{hector.buffiere@ens.fr}
	\author{Eunjung Kim}\address{KAIST, Daejeon, South Korea \\ and CNRS, Paris, France}\email{eunjungkim78@gmail.com}
	\author{Patrice Ossona de Mendez}\address{Centre d'Analyse et de Math\'ematiques Sociales (CNRS, UMR 8557), Paris, France\\ and Computer Science Institute of Charles University, Praha, Czech Republic}\email{pom@ehess.fr}
	\date{}
\begin{abstract}
	\emph{Stability} and \emph{dependence} are model-theoretic notions that have recently proved  highly effective in the study of structural and algorithmic properties of hereditary graph classes, and are considered key notions for generalizing to hereditary graph classes the theory of sparsity developed for monotone graph classes (where an essential notion is that of \emph{nowhere dense} class). The theory of sparsity was initially built on the notion of \emph{shallow minors} and on the idea of excluding different sets of minors, depending on the depth at which these minors can appear.

	In this paper, we follow a similar path, where \emph{shallow vertex minors} replace shallow minors.  In this setting, we provide a neat characterization of stable / dependent hereditary classes of graphs: A hereditary class of graphs $\mathscr C$ is
	\begin{itemize}
		\item 	\emph{dependent} if and only if it does not contain all permutation graphs and, for each integer $r$,  it excludes some split interval graph as a depth-$r$ vertex minor; 
		\item  	\emph{stable} if and only if, for each integer $r$,  it excludes some half-graph as a depth-$r$ vertex minor.
	\end{itemize}
	
	A key ingredient in proving these results is the preservation of stability and dependence of a class when taking bounded depth shallow vertex minors. We extend this preservation result to binary structures and get, as a direct consequence, that bounded depth shallow vertex minors of graphs with bounded twin-width have bounded twin-width.
\end{abstract}
\maketitle

\section{Introduction}
Graph minors are central to structural graph theory and have been extensively studied by Robertson and Seymour in their \emph{Graph minor project}. The shallow version of graph minors then allowed to extend some of the properties of proper minor closed classes of graphs to classes with bounded expansion and nowhere dense classes, and constitutes the foundation of the theory of sparsity initiated by Jaroslav Ne\v set\v ril and the third author \cite{Sparsity}. 

Vertex minors have been introduced by Oum \cite{OUM200579} in his foundational paper on rank-width. Recall that 
a graph $H$ is a \emph{vertex-minor} of a graph $G$ if it can be reached from $G$ by the successive application of local complementations  and vertex deletions.
It appears that vertex minors are a natural dense analog to graph minors as witnessed, for instance, by the grid theorem for vertex minors \cite{GEELEN202393}. 

Shallow vertex minors have been introduced in \cite{modulo}, where it is proved that shallow vertex minors of graphs in a class with bounded expansion can be obtained through  first-order transductions. This notion is defined inductively from the one of depth-$1$ vertex minor. A graph $H$ is a \emph{depth-$1$ vertex-minor} of a graph $G$ if it can be reached from $G$ by  local complementation of the vertices of an independent set and vertex deletions.
In this paper, we give further evidence that  shallow vertex minors are relevant to the structural study of hereditary graph classes.

It is known that a class is \emph{nowhere dense} if it excludes, for each integer $r$, some graph (equivalently, some clique) as a depth-$r$ shallow minor and that, for a monotone class of graphs, being nowhere dense is equivalent to the model theoretical notions of stability and dependence \cite{Adler2013}. Note that classes excluding a vertex minor are not, in general, dependent. Indeed, as noticed in \cite{modulo},  if a class is closed under vertex minors, then it is dependent if and only if it has bounded cliquewidth.

Study of hereditary classes of graphs from the perspective of model theory recently gained a lot of attention, and led to numerous fundamental results both in structural and algorithmic graph theory. Particularly, stability and dependence appeared to play a special role (like in classical model theory). For instance, it has been proved that first-order model checking is fixed-parameter tractable on hereditary stable classes of graphs \cite{MCST}, while it is conjectured that this problem is  fixed-parameter tractable on those hereditary classes of graphs that are dependent. On the dependent but unstable regime, it is known that first-order model checking is fixed-parameter tractable on classes with bounded twin-width \cite{twin-width1}.

The case of hereditary classes of ordered graphs (that is, of graphs with a fixed linear order on the vertex set) has been settled.
\begin{ext_theorem}[\cite{Tww_ordered}]
	\label{thm:twwNIP}
 For a hereditary class of ordered graphs $\mathscr C$, the following are equivalent (assuming ${\sf FPT}\neq{\sf AW}[\ast]$):
\begin{enumerate}[(1)]
	\item First-order model checking is fixed-parameter tractable on $\mathscr C$;
	\item $\mathscr C$ has bounded twin-width;
	\item $\mathscr C$ is dependent;
	\item\label{it:tww_excl} $\mathscr C$ contains neither the class of all ordered permutation graphs, nor any of $24$ special classes encoding ordered matchings.
\end{enumerate}
\end{ext_theorem}

In the presence of a linear order on the vertex set, a matching can be viewed as a family of intervals and thus naturally defines an interval graph with two sorts of intervals, disjoint ``short'' intervals for the vertices, and ``long'' intervals for the matching. In Theorem~\ref{thm:twwNIP}, the $24$ classes are constructed from perfect matchings between two subsets $A$ and $B$ of vertices with $\max A<\min B$. This constraint translates to the property that the interval graph associated to the ordered matching is split (See~\Cref{fig:SI}): its vertex set is the union of an independent set (the vertices) and a clique (the matching). From this point of view, condition~\eqref{it:tww_excl} of Theorem~\ref{thm:twwNIP} excludes the ordered analog of permutation graphs and split interval graphs. 

\begin{figure}[ht]
	\begin{center}
		\includegraphics[width=.75\textwidth]{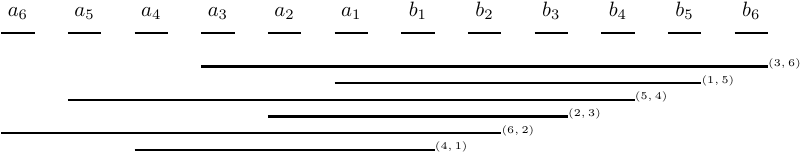}
	\end{center}
	\caption{The split interval graph associated to the matching $M=\{(6,2),(5,4),(4,1),(3,6),(2,3),(1,5)\}$.}
	\label{fig:SI}
\end{figure}

In this paper, we 
provide a full characterization of those hereditary classes of graphs that are  stable or dependent classes in terms of excluded shallow vertex minors.

\begin{theorem}
	\label{thm:NIP_vm}
	Let $\mathscr C$ be a hereditary class of graphs.
	Then, $\mathscr C$ is dependent if and only if the class $\mathscr C$ does not contain all permutation graphs and,
	for every integer $r$, the class $\mathscr C$ excludes some split interval graph as a depth-$r$ shallow vertex minor.
\end{theorem}

This characterization (proved in Section~\ref{sec:char} as \cref{thm:xNIP_vm}) is based on the following graphs:
\begin{itemize}
	\item a \emph{permutation graph} is a graph with vertex set $[n]$ associated to a permutation $\sigma\in\mathfrak S_n$, where $ij$ is an edge if $(i,j)$ is an inversion of $\sigma$ (i.e. if $(i<j)\leftrightarrow (\sigma(i)>\sigma(j))$);
		\begin{center}
		\parbox[c]{.2\textwidth}{\includegraphics[scale=.6]{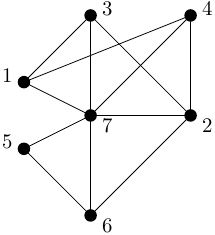}}$\qquad(\sigma=(7\ 3\ 4\ 1\ 6\ 2\ 5))$
	\end{center}
	\item a \emph{split interval} graph is a graph that is both \emph{split} (meaning that is vertex set can be partitioned into an independent set and a clique) and an \emph{interval graph} (i.e. an intersection graph of intervals of the line). See \cite{foldes1977split} for more on split interval graphs.
	\begin{center}
	\includegraphics[scale=.5]{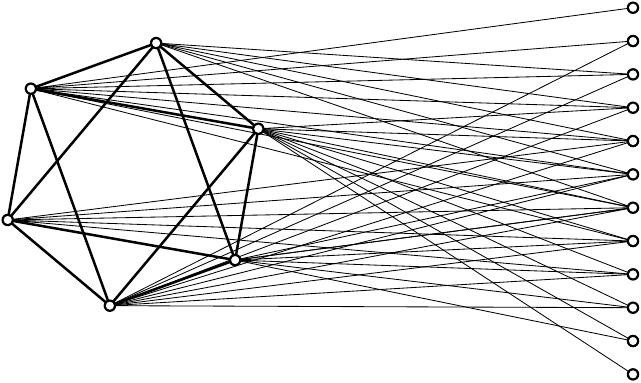}
\end{center}	
\end{itemize}

It is known \cite{RW_SODA} that stable hereditary classes of graphs are exactly dependent hereditary classes of graphs that do not contain the class of all half-graphs. This is a key ingredient in proving  the next characterization.

\begin{theorem}
	\label{thm:stable_vm}
	Let $\mathscr C$ be a hereditary class of graphs.
	Then, $\mathscr C$ is stable if and only if, for every integer $r$, the class $\mathscr C$ excludes some half-graph as a depth-$r$ shallow vertex minor.
\end{theorem}

This characterization (proved in Section~\ref{sec:char} as \cref{thm:xstable_vm}) is based on the following graphs:
\begin{itemize}
		\item the \emph{half-graph} of order $n$ has vertices $a_1,\dots,a_n,b_1,\dots,b_n$ and edges $a_ib_j$ where $i\leq j$;

\begin{center}
	\includegraphics[scale=.6]{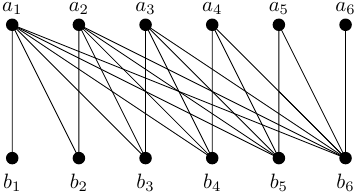}
\end{center}
 \end{itemize}

The paper is organized as follows: In Section~\ref{sec:flip} we give some preliminaries on flips, local complementation, and shallow vertex minors. In Section~\ref{sec:commut} we analyze in what sense flip and local complementation commute.  After recalling the properties of dependence and stability in Section~\ref{sec:MT}, we prove that these properties are preserved by taking shallow vertex minors, and deduce in Section~\ref{sec:char} our characterization theorems. In Section~\ref{sec:bin}, we discuss generalization of the results of Section~\ref{sec:MT} to binary relational structures.

\section{Flips, local complementations, and shallow vertex minors}
\label{sec:flip}
\subsection{Flips}
In this paper, it will be convenient to consider that the adjacency relation $E_G$ of a graph $G$ has value in $\GF2$, with $E_G(u,v)=1$ if $u$ and $v$ are adjacent. This will allow us to use the standard arithmetic operations of $\GF2$.

A \emph{$k$-flip} on a set  $V$ is a pair $F=(\iota,\tau)$, where $\iota:V\rightarrow [k]$ and $\tau:[k]\times [k]\rightarrow \GF2$ is a symmetric function (i.e. $\tau(i,j)=\tau(j,i)$ for all $i,j\in [k]$). Given a flip $F$ on $V(G)$, the graph $G\oplus F$ has the same vertex set as $G$, and its adjacency relation is defined by
\begin{equation}
	\label{eq:flip}
	E_{G\oplus F}(x,y)=E_G(x,y)+\tau(\iota(x),\iota(y)).
\end{equation}

Each flip $F$ on $V(G)$ induces a partition of $V(G)$ into $k$ (possibly empty) parts
$\iota^{-1}(1), \dots, \iota^{-1}(k)$, also called \emph{$F$-classes}. More generally, for a subset $X$ of $V(G)$, the flips $F$ induces a partition of $X$ into parts $\iota^{-1}(1)\cap X, \dots, \iota^{-1}(k)\cap X$, the \emph{$F$-classes of $X$}.

If $G$ is a graph and $F=(\iota,\tau)$ is a $k$-flip on a set $X\supseteq V(G)$, then $F$ naturally defines a $k$-flip on $V(G)$ by considering the restriction of $\iota$ to $V(G)$.
By abuse of notation, we still denote $G\oplus F$ the graph obtained by applying the restricted $k$-flip.


A subset $X$ of vertices of $G$ is \emph{$F$-homogeneous} if the adjacency of any two distinct vertices $u,v\in X$ depends only on $\iota(u)$ and $\iota(v)$. 
A flip $F'$ on $V(G)$ is \emph{$X$-compatible} with $F$ if the (non-empty) traces on $X$ of the parts of $F'$ are the same as the (non-empty) traces on $X$ of the parts of $F$.
It is immediate that $X$ is $F$-homogeneous if and only if there exists a flip $F'$ on $V(G)$ that is $X$-compatible with $F$ and such 
that $X$ is an independent set of $G\oplus F\oplus F'$.

\subsection{Local complementations}
Let $G$ be a graph and let $v$ be a vertex of $G$, the graph obtained from $G$ by \emph{local complementation} of $v$ is the graph $G\ast v$  with same vertex set as $G$, with adjacency relation defined by
\begin{equation}
	\label{eq:lca}
	E_{G\ast v}(x,y)=E_G(x,y)+E_G(x,v)\cdot E_G(v,y).
\end{equation}

If $u$ and $v$ are adjacent,  \emph{pivoting} the edge $uv$ is the operation obtained by successive local complementation of $u,v$, and $u$: $G\wedge uv=G\ast u\ast v\ast u$ ($=G\ast v\ast u\ast v$). This operation flips the edges between the private neighbors of $u$, the private neighbors of $v$, and the common neighbors of $u$ and $v$, and then exchanges $u$ and $v$.
Thus, if $x,y$ are distinct vertices different from $u$ and $v$, it is easily checked that
\begin{equation}
	E_{G\wedge uv}(x,y)=E_G(x,y)+E_G(x,u)\cdot E_G(y,v)+E_G(x,v)\cdot E_G(y,u).
\end{equation}

If $I=\{v_1,\dots,v_k\}$ is an independent set, then
$G\ast v_1\ast\dots\ast v_k=G\ast v_{\sigma(1)}\ast\dots\ast v_{\sigma(n)}$ for every permutation $\sigma$ of $[k]$. This justifies the notation $G\ast I$ for the graph obtained from $G$ by successive local complementation of all the vertices in the independent set $I$.
Note that we will consider the notation $G\ast X$ faulty if $X$ is not independent in $G$. It immediately follows from \eqref{eq:lca} that 
if $I$ is an indepedent set of a graph $G$,  the adjacency relation of $G\ast I$ is
\begin{equation}
	\label{eq:lcI}
	E_{G\ast I}(x,y)=E_G(x,y)+\sum_{v\in I}E_G(x,v)\cdot E_G(v,y).
\end{equation}

A \emph{depth $1$ vertex minor} of a graph $G$ is an induced subgraph of $G\ast I$, for some independent set $I$ of $G$. We denote by $\svm_1(G)$ the set of all the depth $1$ vertex minors of $G$ and, more generally, by $\svm_1(\mathscr C)$ the set of all the depth $1$ vertex minors of graphs in $\mathscr C$. 
We further define inductively \emph{depth $c$ vertex minors} (for positive integers $c>1$)  as depth $1$ vertex minors of depth $(c-1)$ vertex minors. Also, it will be convenient to consider that  depth $0$ vertex minors are simply induced subgraphs. We consequently extend the notations $\svm_1(G)$ and $\svm_1(\mathscr C)$ to depth $c$ vertex minors, as $\svm_c(G)$ and $\svm_c(\mathscr C)$.

For a graph $G$ and a vertex $v$ (resp. a subset $D$ of vertices of $G$), we denote by $G-v$ (resp. $G-D$) the induced subgraph of $G$ obtained by deleting $v$ (resp. $D$). Thus, a depth-$1$ vertex minor of $G$ has the form $G\ast I-D$, for some (possibly empty) independent set $I$ and some (possibly empty) subset $D$ of vertices. It is immediate that if $u$ and $v$ are (distinct) vertices of $G$, then $(G-u)\ast v=(G\ast v)-u$. Thus, for $c\geq 1$, depth-$c$ shallow vertex minors of $G$ are graphs of the form 
$G\ast I_1\ast\dots\ast I_c-D$, where $I_i$ is an independent set of 
$G\ast I_1\ast\dots\ast I_{i-1}$ and $D$ is a subset of vertices. 

\section{Commuting flips and vertex minors}
\label{sec:commut}
Flips and local complementations do not commute in general. However, we shall see that the composition of a flip by a shallow vertex minor operation can be rewritten as the composition of a shallow vertex minor operation and a flip (and vice-versa).

We start our study by a special case, considering the graph $G\oplus F\ast I$, where $I$ is an independent set of both $G$ and $G\oplus F$.

\begin{lemma}
	\label{lem:commute0}
	Let $I$ be an independent set of a graph $G$ and let $F=(\iota,\tau)$ be a $k$-flip on $V(G)$. 
	If $I$ is independent in $G\oplus F$, then there exists a $k2^k$-flip $F'$ on $V(G)$ such that 
	\[
	G\ast I\oplus F'=G\oplus F\ast I. 
	\]
\end{lemma}
\begin{proof}
	We have
	\[
	E_{G\ast I}(u,v)=E_G(u,v)+\sum_{z\in I} E_G(u,z)\cdot E_G(z,v).
	\]
	
	As $I$ is an independent set in $G\oplus F$, 
	\begin{align*}
		E_{G\oplus F\ast I}(u,v)&=E_{G\oplus F}(u,v)+\sum_{z\in I}\bigl(E_{G\oplus F}(u,z)\cdot E_{G\oplus F}(z,v) \bigr)\\
		&=E_G(u,v)+\tau(\iota(u),\iota(v))+\sum_{z\in I}\bigl((E_{G}(u,z)+\tau(\iota(u),\iota(z)))\cdot (E_{G}(v,z)+\tau(\iota(v),\iota(z))) \bigr).
	\end{align*}
	For $i,j\in [k]$ and $u\in V(G)$ we define 
	\begin{align}
	\zeta(i,j)&=\sum_{z\in I} \tau(i,\iota(z))\cdot\tau(j,\iota(z))\\
\varsigma(u,i)&=\sum_{z\in I}E_G(u,z)\cdot \tau(i,\iota(z)). 	\end{align}
	Then, 
	\begin{align*}
		E_{G\oplus F\ast I}(u,v)&= E_{G\ast I}(u,v)+\tau(\iota(u),\iota(v))+\zeta(\iota(u),\iota(v))+\varsigma(u,\iota(v))+\varsigma(v,\iota(u)).
	\end{align*}
	Define $\tau':([k]\times\{0,1\}^k)\times([k]\times\{0,1\}^k)\rightarrow\GF2$ 
	and $\iota':V(G)\rightarrow [k]\times \{0,1\}^{[k]}$ (where $\{0,1\}^{[k]}$ is interpreted as the set of all functions from $[k]$ to $\{0,1\}$) by
	\begin{align*}
		\tau'((i,\alpha),(j,\beta))&=\tau(i,j)+\zeta(i,j)+\alpha(j)+\beta(i)\\
		\iota'(v)&=(\iota(v),\varsigma(v,\,\cdot\,)).
	\end{align*}
	Then $(\tau',\iota')$ defines a $k2^k$-flip $F'$ on $V(G)$ such that 
$E_{G\oplus F\ast I}(u,v)=E_{G\ast I}(u,v)+\tau'(\iota'(u),\iota'(v))$.
\end{proof}

In order to extend this lemma to the  case where $I$ is not assumed to be independent in $G\oplus F$, we start by 
giving a (technical) variant of Lemma~\ref{lem:commute0}. 

\begin{lemma}
	\label{lem:commute0b}
		Let $G$ be a graph, let $F$ be a $k$-flip on $V(G)$,  let $I\subseteq V(G)$ be $F$-homogeneous, in $G$ (hence in $G\oplus F$)  and let 
		$J\subseteq I$ be  non-empty subset of an $F$-class of $J$ that is independent in both $G$ and $G\oplus F$.
		Then, there exists a $k2^k$-flip $F'$ that is $J$-compatible with $F$, such that 
		$G\ast J\oplus F'=G\oplus F\ast J$ and $I\setminus J$ is $F'$-homogeneous in $G\ast J$.
\end{lemma}
\begin{proof}
	As  $J$ is independent in both $G$ and $G\oplus F$ there exists, according to Lemma~\ref{lem:commute0}, a $k2^k$-flip $F'=(\iota',\tau')$ such that 
	\[G\ast J\oplus F'=G\oplus F\ast J.\]

	As $J$ is in a subset of an $F$-class of $I$, we have $J\subseteq\iota^{-1}(i_0)\cap I$ for some  $i_0\in [k]$. 
	As $I$ is $F$-homogeneous, there exists a symmetric function $f:[k]\times [k]\rightarrow\GF2$ such that, for $x,y\in I$ we have $E_G(x,y)=f(\iota(x),\iota(y))$.
	Thus, if $v\in J$ and $i\in [k]$, we have
\begin{align*}
	\varsigma(v,i)&=\sum_{z\in J}E_G(v,z)\cdot \tau(i,\iota(z))\\
	&=\sum_{z\in J} f(\iota(v),i_0)\cdot\tau(i,i_0)\\
	&=|J|\cdot f(\iota(v),i_0)\cdot\tau(i,i_0),
\end{align*}
where $|J|$ is meant as the cardinal of $J$ counted in $\GF2$, that is the parity of $|J|$.
It follows that $\varsigma(v,\,\cdot\,)$ is a function of $\iota(v)$, thus $F$ and $F'$ are $I$-compatible.	

Let $u,v$ be distinct vertices in $I\setminus J$. 
As $I$ is $F$-homogeneous in $G$ we have
\begin{align*}
	E_{G\ast J}(u,v)&=E_G(u,v)+\sum_{z\in J}E_G(u,z)\cdot E_G(z,v)\\
	&=f(\iota(u),\iota(v))+|J|\cdot f(\iota(u),i_0)\cdot f(\iota(v),i_0).
\end{align*}
Thus, $I\setminus J$ is $F'$-homogeneous in $G\ast J$. 
\end{proof}

\begin{lemma}
	\label{lem:commute}
	There exists a function $F(k)$ (basically, a tower of height $k$) such that for every 
	graph $G$, every independent set $I$ of $G$, and every $k$-flip $F=(\iota,\tau)$  on $V(G)$ there exists a partition $J_1,\dots,J_p$ of $I$ into at most $2k$ parts and an $F(k)$-flip $F'$ on $V(G)$ such that 
	\[
	G\ast I\oplus F'=G\oplus F\ast J_1\ast\dots\ast J_p. 
	\]
\end{lemma}
\begin{proof}
	First, note that  $I$ is $F$-homogeneous in $G\oplus F$ as it is independent in $G$. 
	We prove the lemma by induction on the number $k$ of $F$-classes of $I$. 
	We allow $I$ to be empty, so that the base case, $k=0$, obviously holds.
	
	Assume that the lemma when the number of $F$-classes of $I$ is at most $k\geq 0$, and 
	let $I$ be an independent set with $(k+1)$ $F$-classes, denoted $I_1,\dots,I_{k+1}$. 
		
\begin{itemize}
   	\item   Assume $I_1$ is independent in $G\oplus F$.
   	
    Then, there exists, according to Lemma~\ref{lem:commute0b} a $k2^k$-flip $F'$ that is $I$-compatible with $F$ such that 
   	\[
   	G\ast I_1\oplus F'=G\oplus F\ast I_1\]
   	 and $I\setminus I_1$ is $F'$-homogeneous in $G\ast I_1$. 
   	 As $F'$ is $I$-compatible with $F$, $I_2,\dots I_{k+1}$ are the $k$ $F'$-classes of $I\setminus I_1$ in $G\ast I_1$. 
   	 By induction, there exist $J_1,\dots,J_p$ and a flip $F''$ such that
   	 \[
   	 (G\ast I_1)\ast (I\setminus I_1)\oplus F''=(G\ast I_1)\oplus F'\ast J_1\ast\dots\ast J_p.
   	 \]
	Hence,
	\begin{align*}
		G\ast I\oplus F''&=(G\ast I_1)\ast (I\setminus I_1)\oplus F''\\
		&=(G\ast I_1)\oplus F'\ast J_1\ast\dots\ast J_p\\
		&=G\oplus F\ast I_1\ast J_1\ast\dots\ast J_p,
	\end{align*}
	which has the desired form, as $\{I_1,J_1,\dots,J_p\}$ is a partition of $I$ with $p+1\leq 2(k+1)$ parts.
\item Otherwise,  $I_1$ is not independent in $G\oplus F$.

Then, $I_1$ is a clique of size at least two of $G\oplus F$. We split $I_1$ into a singleton $\{a_1\}$ and $I_1\setminus\{a_1\}$.
According to the above, there exists a $k2^k$-flip $F'$ that is $I$-compatible with $F$ such that 
\[G\ast a_1\oplus F'=G\oplus F\ast a_1\]
 and $I\setminus\{a_1\}$ is $F'$-homogeneous in $G\ast a_1$. 
As $F'$ is $I$-compatible with $F$, $I_1\setminus\{a_1\}, I_2,\dots I_{k+1}$ are the $F'$-classes of $I\setminus I_1$. As $I_1$ is a clique in $G\oplus F$, $I_1\setminus\{a_1\}$ is independent in $G\oplus F\ast a_1=G\ast a_1\oplus F'$. Hence, according to Lemma~\ref{lem:commute0b}, there exists a $(k 2^k)2^{k2^k}$-flip $F''$ that is $I\setminus\{a_1\}$-compatible with $F'$ such that 
\[(G\ast a_1)\ast (I_1\setminus\{a_1\}) \oplus F''
=(G\ast a_1)\oplus F'\ast (I_1\setminus \{a_1\}).\]
 Moreover, as $F''$ is $I\setminus\{a_1\}$-compatible with $F'$, $I_2,\dots I_p$ are the $k$ $F''$-classes of $I\setminus I_1$ in $G\ast I_1$.

By induction, there exist $J_1,\dots,J_p$ and a flip $F'''$ such that
\[
(G\ast I_1)\ast (I\setminus I_1)\oplus F'''=(G\ast I_1)\oplus F''\ast J_1\ast\dots\ast J_p.
\]
Hence,
\begin{align*}
	G\ast I\oplus F'''&=(G\ast I_1)\ast (I\setminus I_1)\oplus F'''\\
	&=(G\ast I_1)\oplus F''\ast J_1\ast\dots\ast J_p\\
	&=(G\ast a_1)\ast (I_1\setminus\{a_1\})\oplus F''\ast I_1\ast J_1\ast\dots\ast J_p\\
	&=(G\ast a_1)\oplus F'\ast (I_1\setminus\{a_1\})\ast J_1\ast\dots\ast J_p\\
	&=G\oplus F\ast a_1\ast (I_1\setminus\{a_1\})\ast J_1\ast\dots\ast J_p,\\
\end{align*}
which has the desired form, as $\{\{a_1\}, (I_1\setminus\{a_1\}), J_1,\dots,J_p\}$ is a partition of $I$ with $p+2\leq 2(k+1)$ parts.
\end{itemize}
\end{proof}
\begin{corollary}
	Let $F=(\iota,\tau)$ be a $k$-flip of $G$ and let  $I$ be an independent set of  $G\oplus F$. 
	Then, there exists a partition $I_1,\dots,I_p$ of $I$ into at most $2k$ parts and an $F(k)$-flip $F'$ on $V(G)$ such that 
	\[
	G\ast I_1\ast\dots\ast I_p\oplus F'=G\oplus F \ast I. 
	\]
\end{corollary}
\begin{proof}
	Applying Lemma~\ref{lem:commute} to $G\oplus F$, we obtain a  partition $I_1,\dots,I_p$ of $I$ into at most $2k$ parts and an $F(k)$-flip $F'$ on $V(G)$ such that 
	\[
	(G\oplus F)\ast I\oplus F'=(G\oplus F)\oplus F\ast I_1\ast\dots\ast I_p. 
	\]
	As flips are involutive, we deduce
	$G\ast I_1\ast\dots\ast I_p\oplus F'=G\oplus F\ast I$.
\end{proof}
\begin{corollary}
	Let $F=(\iota,\tau)$ be a $k$-flip of $G$ and let  $I$ be an independent set of  $G$. 
	Then, there exists a partition $I_1,\dots,I_p$ of $I$ into at most $2k$ parts and an $F(k)$-flip $F'$ on $V(G)$ such that 
	\[
	G\oplus F'\ast I_p\ast\dots\ast I_1=G\ast I\oplus F. 
	\]
\end{corollary}
\begin{proof}
	Note that $I$ is an independent set of $G\ast I$ as well.
	Applying Lemma~\ref{lem:commute} to $G\ast I$, we obtain a  partition $I_1,\dots,I_p$ of $I$ into at most $2k$ parts and an $F(k)$-flip $F'$ on $V(G)$ such that 
	\[
	(G\ast I)\ast I\oplus F'=(G\ast I)\oplus F\ast I_1\ast\dots\ast I_p. 
	\]
	As local complementations are involutive, we deduce
	$G\oplus F'\ast I_p\ast \dots\ast I_1=G\ast I\oplus F$.
\end{proof}

\begin{lemma}
	\label{lem:clean}
	Let $I$ be an independent set of a graph $G$ and let $F=(\iota,\tau)$ be a $k$-flip on $V(G)$. 
	Then, there exists a $2k$-flip $F'$ on $V(G)$ such that
	$G\oplus F'$ is obtained from $G\oplus F$ by removing all the edges between vertices of $I$.
\end{lemma}
\begin{proof}
	Define $\iota':V(G)\rightarrow [k]\times\{0,1\}$ by $\iota'(v)=(\iota(v),1)$ if $v\in I$ and $\iota'(v)=(\iota(v),0)$, otherwise.
	Define $\tau'$ by $\tau'((i,a),(j,b))=\tau(i,j)$ if $(a,b)\neq (1,1)$ and $\tau'((i,a),(j,b))=0$, otherwise. Let $F'=(\tau',\iota')$.
	Then, if $u$ and $v$ do not both belong to $I$, we have
	$E_{G\oplus F'}(u,v)=E_{G\oplus F}(u,v)$. However, if both $u$ and $v$ belong to $I$, we have $E_{G\oplus F'}(u,v)=E_G(u,v)=0$.
\end{proof}

\begin{lemma}
	\label{lem:local}
	Let $H\in\svm_1(G)$ and let $x,y\in V(H)$. 
	Then $\dist_H(x,y)\geq \frac12\,\dist_G(x,y)$.
\end{lemma}
\begin{proof}
	Let $H=G\ast I- D$.
	Assume $\dist_H(x,y)=1$. Then, either $\dist_G(x,y)=1$ or some edge between $x$ and
	$y$ has been created by the local complementation of the vertices in $I$. In this
	later case, $x$ and $y$ have a common neighbor in $I$ hence $\dist_G(x,y)=2$.
	Generally, considering a shortest path between $x$ and $y$ in $H$, we deduce $\dist_H(x,y)\geq \frac12\,\dist_G(x,y)$.
\end{proof}

As a direct consequence of Lemmas~\ref{lem:commute0}, \ref{lem:clean} and~\ref{lem:local} we have
\begin{lemma}
	\label{lem:spread}
	Let $F$ be a $k$-flip on $V(G)$ and let  $I$ be an independent set of  $G$.  Then, there exists a $2k 2^{2k}$-flip $F'$ on 
	$V(G)$ such that, for every two vertices $x,y$ in $V(G)$ we have
\begin{equation}
	\dist_{G\ast I\oplus F'}(x,y)\geq \frac12\dist_{G\oplus F}(x,y).
\end{equation}
\end{lemma}
\begin{proof}
	According to Lemma~\ref{lem:clean} there exists a $2k$-flip $\widehat F$ such that $\dist_{G\oplus \widehat F}\geq \dist_{G\oplus F}$ and $I$ is independent in $G\oplus \widehat F$. By Lemma~\ref{lem:commute0} there exists a $(2k) 2^{2k}$-flip $F'$ such that $G\ast I\oplus F'=G\oplus \widehat F\ast I$. According to Lemma~\ref{lem:local}, for every $x,y\in V(G)$, we have
\[
\dist_{G\ast I\oplus F'}(x,y)=\dist_{G\oplus \widehat F\ast I}(x,y)\geq \frac12\dist_{G\oplus \widehat F}(x,y) \geq \frac12\dist_{G\oplus F}(x,y).
\]
	
\end{proof}

\section{Dependence and Stability}
\label{sec:MT}
The notions of dependence (or NIP) and stability are central to the classification theory in Model Theory.
Let $\mathscr C$ be a class of graphs and let  $\varphi(\bar x;\bar y)$ be a first-order partitioned formula, that is a first-order formula whose free variables are partitioned into two tuples (here $\bar x$ and $\bar y$).

The formula $\varphi$ is \emph{unstable} on $\mathscr C$ if, for every integer $n$ there exists a graph $G\in\mathscr C$ and tuples $\bar a_1,\dots,\bar a_n,\bar b_1,\dots,\bar b_n$ of vertices of $G$ such that $G\models \varphi(\bar a_i,\bar b_j)$ if and only if $i\leq j$. A class $\mathscr C$ is \emph{stable} if no formula is unstable on $\mathscr C$.

The formula $\varphi$ is \emph{independent} on $\mathscr C$ if, for every integer $n$ there exists a graph $G\in\mathscr C$ and tuples   $\bar a_1,\dots,\bar a_n,\bar b_\emptyset,\dots,\bar b_{[n]}$ of vertices of $G$ such that $G\models \varphi(\bar a_i,\bar b_J)$ if and only if $i\in J$. A class $\mathscr C$ is \emph{dependent} (or {\sf NIP}) if no formula is independent on $\mathscr C$.

We illustrate this notion by an example:
\begin{example}
	\label{ex:si}
	The class of split interval graphs is independent.
\end{example}
\begin{proof}
	Consider the following formulas.
\begin{align*}
	\nu(x)&:= \exists y\ \bigl(\neg(x=y)\wedge\forall z\ (E(x,z)\leftrightarrow E(y,z)\bigr)\\
	\eta(x)&:=\neg\nu(x)\wedge\forall y\ (\nu(y)\rightarrow\neg E(x,y))\\
	\mu_\nu(x,y)&:=E(x,y)\wedge\forall z\ \bigl(((\nu(z)\wedge E(z,y))\rightarrow (\forall t\ E(x,t)\rightarrow E(z,t))\bigr)\\
	\mu_\eta(x,y)&:=E(x,y)\wedge\forall z\ \bigl(((\eta(z)\wedge E(z,y))\rightarrow (\forall t\ E(x,t)\rightarrow E(z,t))\bigr)\\
	\varphi(x,y)&:=	\exists z\ (\mu_\nu(x,z)\wedge\mu_\eta(y,z))
\end{align*}
These formulas will be used on the following construction of a split interval graph $G$ (See~\Cref{fig:power}). Let $n>1$ be an integer. We consider $2n+2^n$ non-intersecting intervals $a_1,a_1',\dots,a_n,a_n',b_{[n]}$, $\dots,b_{\{1\}},b_\emptyset$ in this order (the subsets of $[n]$ being ordered in reverse lexicographic order). We add an interval from $a_1$ to $a_n'$ then, for each $i\in [n]$ and each $J\subseteq [n]$ containing $i$, we add an interval from $a_i$ to $b_J$.
The formula $\nu(x)$ expresses that $x$ has a false twin. Hence, $\nu(G)=\{a_1,\dots,a_n,a_1',\dots,a_n'\}$. The formula $\eta$ allows defining the remaining of the stable part of $G$: $\eta(G)=\{b_\emptyset,\dots,b_{\{1,\dots,n\}}\}$. The formula $\mu_\nu(x,y)$ expresses that $x$ is the leftmost interval in $\nu(G)$ that is adjacent to $y$, while $\mu_\eta(x,y)$ expresses that $x$ is the rightmost interval in $\eta(G)$ that is adjacent to $y$. Finally, we easily check that $G\models\phi(a_i,b_J)$ if and only if $i\in J$. Hence, $\phi$ is independent on the class of split interval graphs.
\end{proof}

\begin{figure}[ht]
	\centering
	\includegraphics[width=.75\textwidth]{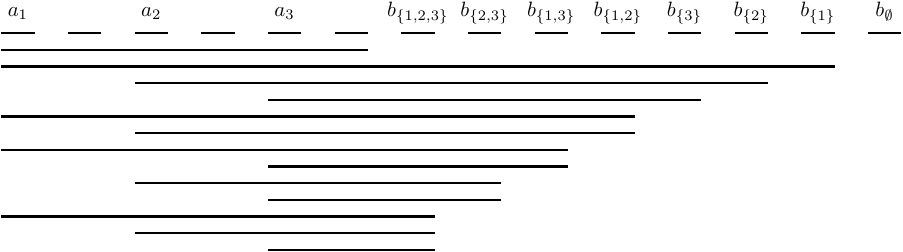}
	\caption{Encoding a power graph in a split interval graph.
		We have $\varphi(a_i,b_J)\leftrightarrow i\in J$.}
	\label{fig:power}
\end{figure}

We can also consider monadic expansions of graphs, that is, relational structures with a single binary relation and finitely many 
unary predicates, which can be seen as vertex-colored graphs.
A class $\mathscr C'$ is a \emph{monadic expansion} of a class of graphs $\mathscr C$ is the class $\mathscr C$ is obtained from $\mathscr C'$ by ``forgetting'' the unary predicates.
A class $\mathscr C$ is \emph{monadically stable} (resp. \emph{monadically dependent})
if every monadic expansion of $\mathscr C$ is stable (resp. dependent). For hereditary classes of graphs, the monadic and non-monadic version collapse:

\begin{ext_theorem}[{\cite{braunfeld2022existential}}]
	\label{thm:collapse}
	Let $\mathscr C$  be a hereditary class of graphs.
	Then $\mathscr C$ is dependent if and only if it is monadically dependent, and it is stable if and only if it is monadically stable.
\end{ext_theorem}

As a direct consequence, we have
\begin{corollary}
	Let $\mathscr C$ be a class of graphs.
	Then $\mathscr C$ is monadically dependent if and only if  the hereditary closure of $\mathscr C$ is dependent, and it is monadically stable if and only if the hereditary closure of $\mathscr C$ is stable.
\end{corollary}

Recently, several characterization theorems have been given for stable and dependent hereditary classes. We shall make use of two types of characterizations.

The first type of characterizations is based on the possibility to push vertices far away by means of a flip.

\begin{ext_theorem}[\cite{dreier2023indiscernibles}]
	\label{thm:flat}
	A class $\mathscr C$ is monadically stable if and only if for every radius $r$,
	there exists an integer $k$ and an unbounded non-decreasing function  $U:\mathbb
	N\rightarrow \mathbb N$ such that for every $G\in\mathscr C$ and every $A\subseteq
	V(G)$ there exists a $k$-flip $F$ and a subset $S\subseteq A$ with $|S|\geq U(|A|)$
	and any two vertices in $S$ are pairwise at distance at least $r$ in $G\oplus F$.
\end{ext_theorem}

\begin{ext_theorem}[{\cite[Theorem 1.3]{dreier2024flipbreakability}}]
	\label{thm:break}
	A class $\mathscr C$ is monadically dependent if and only if for every radius $r$, there exists an integer $k$ and an unbounded non-decreasing function  $U:\mathbb N\rightarrow \mathbb N$ such that for every $G\in\mathscr C$ and every $A\subseteq V(G)$ there exists a $k$-flip $F$ and subsets $A_1,A_2\subseteq A$ with $|A_1|,|A_2| \geq U(|A|)$ and any vertex in $A_1$ is at distance at least $r$ in $G\oplus F$ to vertices in $A_2$.
\end{ext_theorem}

The second type of characterization is structural.

\begin{ext_theorem}[{\cite[Theorem 1.6]{dreier2024flipbreakability}}]
	\label{thm:DMT}
	Let $\mathscr C$ be a graph class. Then $\mathscr C$ is monadically dependent if and only if for every $r\geq  1$ there exists $n\in\mathbb N$ such that $\mathscr C$ excludes as induced subgraphs
	\begin{itemize}
		\item 	all flipped star $r$-crossings of order $n$, and
		\item all flipped clique $r$-crossings of order $n$, and
		\item all flipped half-graph $r$-crossings of order $n$, and
		\item  the comparability grid of order $n$.
	\end{itemize}
\end{ext_theorem}

We take time to define the families of graphs used in Theorem~\ref{thm:DMT}.~
Let $r$ be a positive integer.
\begin{itemize}
	\item  A \emph{star $r$-crossings of order $n$}  (\Cref{sfig:s}) is the $r$-subdivision of $K_{n,n}$. It consists of principle vertices  $a_1,\dots,a_n$ and $b_1,\dots,b_n$ and internally vertex-disjoint  paths $P_{i,j}$ ($i,j\in [n]$) with vertices $p_{i,j,k}$ (with $0\leq k\leq r+1$ monotone on the path), where $p_{i,j,0}=a_i$ and $p_{i,j,r+1}=b_j$.
\item 
A \emph{clique $r$-crossings of order $n$} (\Cref{sfig:c}) is the graph obtained from a star $r$-crossings of order $n$ by making the neighborhood of each principal vertex complete. 
\item 
A \emph{half-graph $r$-crossings of order $n$} (\Cref{sfig:h}) is the graph obtained from a star $r$-crossings of order $n$ by making adjacent 
the vertices $a_i$ and $p_{i',j,1}$ whenever $i'\geq i$  and the vertices $b_j$ and $p_{i,j',r}$ whenever $j'\geq j$.
\end{itemize}
A \emph{flipped} star $r$-crossings (resp. clique $r$-crossing, half-graph $r$-crossings)
of order $n$ is the graph obtained from a star $r$-crossings
(resp. a clique $r$-crossing, a half-graph $r$-crossings)
 of order $n$ by applying a flip based on the partitions with parts
$\{p_{i,j,k}\colon i,j\in [n]\}$, for $0\leq k\leq r+1$, which we call the flip-parts.
\begin{itemize}
	\item 
Finally, the \emph{comparability grid} of order $n$ (\Cref{sfig:comp}) has vertex set $\{a_{i,j}\colon i,j\in [n]\}$, with 
$a_{i,j}$ adjacent to $a_{i',j'}$ whenever $i=i'$, $j=j'$, or $(i<j)\leftrightarrow (i'<j')$. 
\end{itemize}

\begin{figure}[ht]
\phantom{A}\hfill
\subfloat[The comparability grid\label{sfig:comp}]{\includegraphics[height=3cm]{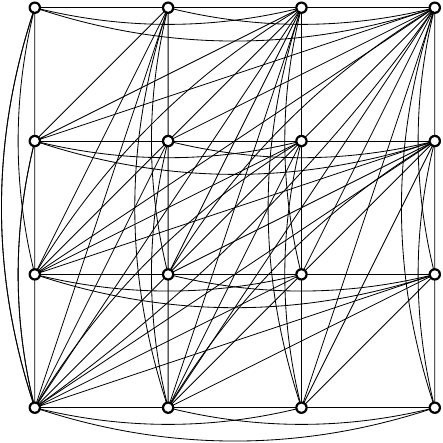}}
\hfill
\subfloat[A (flipped) star $r$-crossing\label{sfig:s}]{\includegraphics[height=3cm]{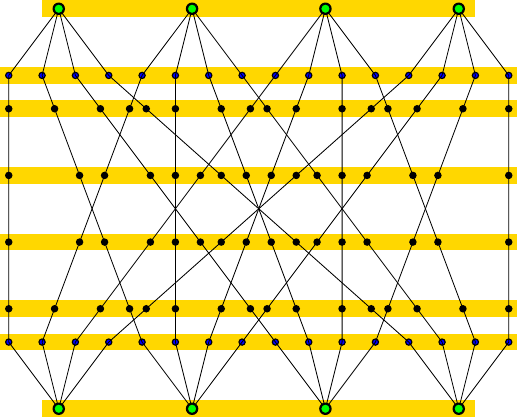}}\hfill\phantom{A}\\
\phantom{A}\hfill
\subfloat[A (flipped) clique $r$-crossing\label{sfig:c}]{\includegraphics[height=3cm]{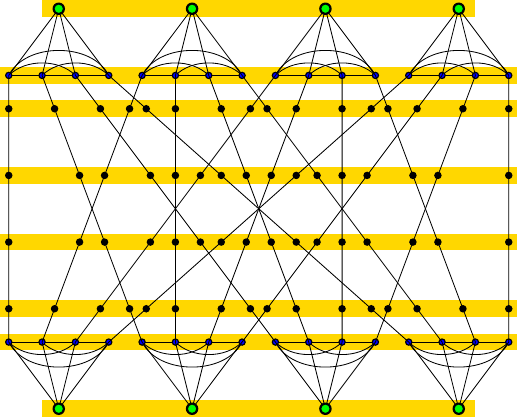}}\hfill
\subfloat[A (flipped) half-graph $r$-crossing\label{sfig:h}]{\includegraphics[height=3cm]{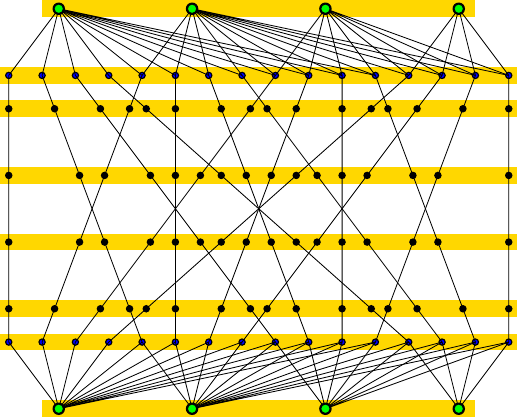}}\hfill\phantom{A}
\caption{The forbidden induced configurations. \Cref{sfig:s,sfig:c,sfig:h}: the flipped versions are obtained by applying an $r+2$-flip whose parts are materialized as horizontal lines of vertices.}
\end{figure}

The next easy lemma will be useful.

\begin{lemma}
	\label{lem:svm_flip}
	Let $G$ be a graph, let $F=(\iota,\tau)$ be a $k$-flip on $V(G)$,
	and let $I$ be an independent set of $G$ containing exactly one element of each $F$-class.
	
	Then, there exist $z_1,\dots,z_p$ in $I$ with $p\leq 3k/2$ (where each element of $I$ is used at most twice), such that
	\begin{equation}
		G\oplus F\ast z_1\ast\dots\ast z_p-N_G[I]=G-N_G[I],
	\end{equation}
	where $N_G[I]$ denotes the closed neighborhood of $I$ in $G$.
\end{lemma}
\begin{proof}
	Let $X=\{i\in [k]\colon \exists j\in [k], \tau(i,j)=1\}$. We prove a strengthening of the lemma statement by induction on $|X|$, where we require that $\iota(z_i)\in X$ for all $i\in [p]$.
	
	The base case is $X=\emptyset$, in which case $G\oplus F=G$, so we can let $p=0$.
	
	Assume we have proved the induction hypothesis for $|X|\leq\ell$ (where $\ell\geq 0$) and let $|X|=\ell+1$.
	
	Assume there exists $a\in X$ with $\tau(a,a)=1$, and let 
	$z$ be the element in $I$ with $\iota(z)=a$.
	Then, for $u,v\notin N_G[z]$,
	\begin{align*}
		E_{G\oplus F\ast z}(u,v)&=E_G(u,v)+\tau(\iota(u),\iota(v))+
		((E_G(u,z)+\tau(\iota(u),a)\cdot ((E_G(v,z)+\tau(\iota(v),a))\\
		&=E_G(u,v)+\tau(\iota(u),\iota(v))+\tau(\iota(u),a)\cdot \tau(\iota(v),a)\\
		&=E_{G\oplus F'}(u,v),
	\end{align*}
	where $F'=(\iota, \tau')$ and $\tau'(i,j)=\tau(i,j)+\tau(i,a)\cdot \tau(j,a)$. In particular, $\tau'(a,i)=0$ for all $i\in [k]$.
	Thus, the result follows from the induction hypothesis.
	
	Otherwise, let $a\in X$. As $\tau(a,a)=0$, there exists $b\neq a$ in $X$ with $\tau(a,b)=1$. Note that $\tau(b,b)=0$.
	Let $z$ (resp. $z'$) be the element of $I$ with 
	$\iota(z)=a$ (resp. $\iota(z')=b$).
	Then, $zz'$ is an edge of $G\oplus F$, and we have, for $u,v\notin N_G[\{z,z'\}]$,
	
	\begin{align*}
		E_{G\oplus F\ast z\ast z'\ast z}(u,v)&=
		E_{G\oplus F\wedge zz'}(u,v)\\
		&=E_G(u,v)+\tau(\iota(u),\iota(v))+
		(E_G(u,z)+\tau(\iota(u),\iota(z))\cdot 	(E_G(v,z')+\tau(\iota(v),\iota(z'))\\
		&\qquad+(E_G(u,z')+\tau(\iota(u),\iota(z'))\cdot 	(E_G(v,z)+\tau(\iota(v),\iota(z))\\
		&=E_G(u,v)+\tau(\iota(u),\iota(v))+
		\tau(\iota(u),a)\cdot\tau(\iota(v),b)+\tau(\iota(u),b)\cdot\tau(\iota(v),a)\\
		&=E_{G\oplus F'}(u,v),
	\end{align*}
	where $F'=(\iota, \tau')$ and $\tau'(i,j)=\tau(i,j)+\tau(i,a)\cdot \tau(j,b)+\tau(i,b)\cdot\tau(j,a)$. In particular, 
	$\tau'(a,i)=0$ and $\tau'(b,i)=0$ for all $i\in [k]$.
	Thus, the statement follows from the induction hypothesis.
\end{proof}

The next lemma shows how to reduce subdivisions.

\begin{lemma}
	\label{lem:sub}
	Let $G$ be a subdivision of a graph $H$, where every edge is subdivided at most $r$ times.
	Then,  $H$ is a depth-$\lceil \log_2 (r+1)\rceil$  vertex minor of $G$.
\end{lemma}
\begin{proof}
	Let $S=V(G)\setminus V(H)$ be the set of all the subdivision vertices of $G$, and let $I\subseteq S$ be an independent set of maximal possible cardinal.	
	Note that if an edge $uv$ of $H$ is subdivided $k\leq r$ times in $G$, then $I$ contains $\lceil k/2\rceil$ of these subdivision vertices, and thus $G\ast I-I$ is a subdivision of $H$, where each edge is subdivided at most $\lfloor r/2\rfloor$ times.
	By induction, $H$ is a depth-$c$ vertex minor of $G\ast I-I$, where
	$2^{c-1}+1\leq \lfloor r/2\rfloor+1\leq 2^{c}$, i.e. 
	if $2^{c}+1\leq r+1\leq  2^{c+1}$. 
	Hence, $H$ is a depth-$\lceil \log_2 (r+1)\rceil$ vertex minor of $G$.
\end{proof}

It will be helpful to consider the following graphs instead of split interval graphs: the \emph{ordered-matching graph} associated to a matching $M\subset [n]\times [n]$ is the graph with vertex set
$\{a_i\colon i\in [n]\}\cup\{b_j\colon j\in [n]\}\cup M$, where $(k,\ell)\in M$ is adjacent to $a_i$ if $i\leq k$ and to $b_j$ if $j\leq \ell$ (See~\Cref{fig:OM}).

\begin{figure}[ht]
\begin{center}
\includegraphics[scale=.75]{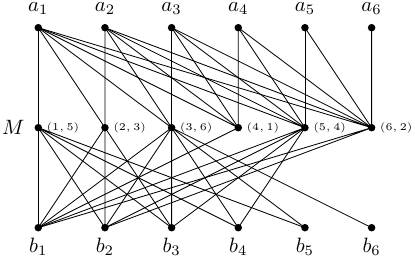}
\end{center}
\caption{Ordered-matching graph associated to\\  $M=\{(1,5), (2,3), (3,6), (4,1), (5,4), (6,2)\}$.}
\label{fig:OM}
\end{figure}

\begin{lemma}
	\label{lem:om2si}
	Every split interval graph is a depth-$1$  vertex minor of an ordered matching graph.
\end{lemma}
\begin{proof}
	Let $H$ be a split interval graph. We construct a supergraph $G$ of $H$ with additional properties. We first let $G=H$. As $G$ is a split interval graph, its vertex set is the disjoint union of a  clique $K$ and an independent set $I$. It is known \cite[Proposition~1]{foldes1977split} that $G$ has an interval representation where every interval in $I$ is a singleton. Hence, the interval representation defines a linear order $<$ on $I$.
	As $K$ is a clique, $\bigcap K$ is not empty. Adding, if necessary,  a vertex to $I$, we can assume that there exists $a_1\in I\cap\bigcap K$. Moreover, by slightly extending some intervals and adding some vertices in $I$, we can assume that all the intervals have distinct leftmost incidence in $I$ and that distinct rightmost incidence in $I$. By adding new intervals and vertices in $I$ if necessary, we can further ensure that every vertex $v\in I$ is the leftmost or the rightmost incidence in $I$ of exactly one interval in $K$ (and that all the intervals in $K$ are incident to $a_1$). Then, the elements of $I$ can be labeled 
	$a_m<\dots<a_1<b_1<\dots<b_m$, and the elements in $K$ can be labeled as a set $M$ of pairs $(i,j)$, where the pair associated to $v\in L$ is $(i,j)$ if $a_i$ is the leftmost incidence of $v$ in $I$, $b_j$ the rightmost incidence of $v$ in $I$. Note 
	that $M$ is a perfect matching of $\{a_1,\dots,a_m\}$ and $\{b_1,\dots,b_m\}$ by construction.
	
	The graph $G'$ obtained by flipping $K$ (i.e. turning $K$ into an independent set) is an ordered-matching graph $H'$.
	Now, $H$ is an induced subgraph of $G'=H'\ast a_1$, hence a depth-$1$ vertex minor of an ordered-matching graph.
	\end{proof}

With Lemma~\ref{lem:svm_flip} in hand, we reduce the different cases of Theorem~\ref{thm:DMT} in terms of shallow vertex minors.
\begin{lemma}
	Let $\mathscr C$ be a class of graphs and let $r$ be a positive integer. Assume that for arbitrarily large integer $n$ the class $\mathscr C$ includes a flipped star $r$-crossings or a flipped clique $r$-crossings of order $n$. Then, $\svm_{c}(\mathscr C)$ is the class of all graphs, where $c=3r/2+4+\lceil \log_2 (2r+1)\rceil$.
\end{lemma}
\begin{proof}
	Consider a flipped star $r$-crossings or a flipped clique $r$-crossings $G$ of order $n$. 
	Let $I=\{p_{1,1,2k}\colon 0\leq 2k\leq r+1\}\cup \{p_{2,2,2k+1}\colon 0\leq 2k+1\leq r+1\}$.
Then, $I$ is an independent set and, according to Lemma~\ref{lem:svm_flip}, a star $r$-crossings or a clique $r$-crossings of order $n-2$ is a depth-$(3(r+2)/2)$ vertex minor of $G$. if the obtained graph is a clique $r$-crossings of order $n-2$, then it can be turned into a star $r$-crossings of order $n-2$ by local complementation of the (independent) set of all its principal vertices. Thus, we get  a star $r$-crossings of order $n-2$ as a depth-$(3r/2+4)$ vertex minor of $G$. This graph contains the $(2r+1)$-subdivision of all the $(2r+1)$-subdivisions of the graphs with order at most $\sqrt{n-2}$. The result follows.
\end{proof}

\begin{lemma}
		Let $\mathscr C$ be a class of graphs and let $r$ be a positive integer. Assume that for arbitrarily large integer $n$ the class $\mathscr C$ includes a flipped half-graph $r$-crossings. Then, $\svm_{c}(\mathscr C)$ includes all order-matching graphs, where $c=3r/2+4+\lceil\log_2 r\rceil$.
\end{lemma}
\begin{proof}
		Consider a flipped half-graph $r$-crossings $G$ of order $2n+3$. 
	Let $I=\{p_{1,1,2k}\colon 2\leq 2k\leq r\}\cup\{p_{2,2,2k+1}\colon 0\leq 2k+1\leq r+1\}\cup\{a_3\}$.
	Then, $I$ is an independent set and, according to Lemma~\ref{lem:svm_flip},  a half-graph $r$-crossings of order $2n$ is a depth-$(3(r+2)/2)$ vertex minor of $G$. If $r=1$, then we get all order-matching graphs of order $2n$ as induced subgraphs of the so-obtained graph.
	If $r\geq 2$, we get, as a depth-$\lceil \log_2 r\rceil$ vertex minor of the so-obtained graph a graph formed by two half graphs and an arbitrary perfect matching between their upper parts.
	By local complementation of one of these upper part (and as the order  $2n$ is even) we get  all order-matching graphs of order $2n$.
\end{proof}

As a consequence of Theorem~\ref{thm:DMT}, the above lemmas, and the fact that every permutation graph is an induced subgraph of a comparability grid (See, for instance, \cite{GEELEN202393}), we get
\begin{corollary}
	\label{cor:svm2NIP}
	Let $\mathscr C$ be a hereditary class of graphs. If $\mathscr C$ is independent, then either $\mathscr C$ includes all permutation graphs or there exists a non-negative integer $r$ such that $\svm_r(\mathscr C)$ includes all split interval graphs.
\end{corollary}

\section{Preservation of dependence and stability by shallow vertex minors}
\label{sec:preserv}

\begin{lemma}
	\label{lem:svm_st}
	Let $\mathscr C$ be a hereditary class of graphs. Then $\mathscr C$  is stable if and only if  $\svm_1(\mathscr C)$ is stable.
\end{lemma}
\begin{proof}
	If $\svm_1(\mathscr C)$ is stable, then $\mathscr C$ is stable as $\mathscr C\subseteq \svm_c(\mathscr C)$.
	
	Assume $\mathscr C$ is stable.
	Let $r$ be a positive integer, let $G\in\mathscr C$, let $I$ be an independent set of $G$ and let $A\subseteq V(G)$.
	According to Theorem~\ref{thm:flat}, there exists a $k_{2r}$-flip $F_{2r}$ and a subset $S$ of $A$ with size at least 
	$U_{2r}(|A|)$, such that the vertices of $S$ are $2r$-independent in $G\oplus F_{2r}$.
	According to Lemma~\ref{lem:spread}, there exists a  $2k_{2r}2^{2k_{2r}}$-flip $F'$ such that 
	$\dist_{G\ast I\oplus F'}\geq \frac12\dist_{G\oplus F}$.
	The integer $k_r'=2k_{2r}2^{2k_{2r}}$, the flip $F_r'=F'$, and the function $U_r'=U_{2r}$ witness that, according to Theorem~\ref{thm:flat}, the class $\{G\ast I\colon G\in\mathscr C, I\text{ independent in }G\}$ is monadically stable. It follows that the hereditary closure $\svm_1(\mathscr C)$ of this class is stable.
\end{proof}

Similarly, we get.

\begin{lemma}
	\label{lem:svm_dep}
	Let $\mathscr C$ be a hereditary class of graphs. Then,  $\mathscr C$  is dependent if and only if 
 $\svm_1(\mathscr C)$ is dependent.
\end{lemma}
\begin{proof}
	If $\svm_1(\mathscr C)$ is dependent, then $\mathscr C$ is dependent as $\mathscr C\subseteq \svm_c(\mathscr C)$.

Assume $\mathscr C$ is dependent.	
	Let $r$ be a positive integer, let $G\in\mathscr C$, let $I$ be an independent set of $G$ and let $A\subseteq V(G)$.
	According to Theorem~\ref{thm:break}, there exists a $k_{2r}$-flip $F_{2r}$ and subset $A_1,A_2$ of $A$ with size at least 
	$U_{2r}(|A|)$, such that $\dist_{G\oplus F_{2r}}(A_1,A_2)>r$.
	According to Lemma~\ref{lem:spread}, there exists a $2k_{2r} 2^{2k_{2r}}$-flip $F'$ such that 
	$\dist_{G\ast I\oplus F'}\geq \frac12 \dist_{G\oplus F}$.
	The integer $k_r'=4k_{2r}$, the flip $F_r'=F'$, and the function $U_r'=U_{2r}$ witness that, according to Theorem~\ref{thm:break}, the class $\{G\ast I\colon G\in\mathscr C, I\text{ independent in }G\}$ is monadically dependent. It follows that the hereditary closure $\svm_1(\mathscr C)$ of this class is dependent	
\end{proof}

\section{The characterization theorems}
\label{sec:char}
\begin{theorem}[Restatement of Theorem~\ref{thm:NIP_vm}]
	\label{thm:xNIP_vm}
	Let $\mathscr C$ be a hereditary class of graphs.
	Then, $\mathscr C$ is dependent if and only if the class $\mathscr C$ does not contain all permutation graphs and,
	for every integer $r$, the class $\mathscr C$ excludes some split interval graph as a depth-$r$ shallow vertex minor.
\end{theorem}
\begin{proof}
	Assume $\mathscr C$ is independent. Then, according to Corollary~\ref{cor:svm2NIP}, either $\mathscr C$  contains all permutation graphs or there exists a non-negative integer $r$ such that $\svm_r(\mathscr C)$ includes all split interval graphs.
	
	Conversely, assume that either $\mathscr C$  contains all permutation graphs or there exists a non-negative integer $r$ such that $\svm_r(\mathscr C)$ includes all split interval graphs. In the first case, $\mathscr C$ is independent, as it is well known that the class of all permutation graphs is independent\footnote{
	One way to see this is to check that the class of all permutations encoded as two linear orders (which is known to be independent \cite{lmcs_perm}) is a transduction of the class of all permutation graphs. To see this, we associate to $\sigma\in\mathfrak S_n$ the permutation $\hat\sigma\in\mathfrak S_{2n}$ defined by $\hat\sigma(2i+1)=\sigma(i)$ and $\hat\sigma(2i)=n+i$.
	Let $A=\{2i\colon i\in [n]\}$ and $B=\{2i+1\colon i\in [n]\}$, then $\sigma$ is isomorphic to the permutation on $A$ defined by the total orders $(x<_1 y):= \bigl(\forall z\ B(z)\rightarrow\bigl(E(x,z)\rightarrow E(y,z)\bigr)\bigr)$ and
	$(x<_2y):= (x<_1 y)\nleftrightarrow E(x,y)$.
	 }
	
	In the latter case, it follows from the fact that the class of split interval graphs is independent (See Example~\ref{ex:si}) that $\svm_r(\mathscr C)$ is independent. Then, according to Lemma~\ref{lem:svm_dep}, the class $\mathscr C$ is independent.
\end{proof}

The characterization of stability will make use of the following result.
\begin{ext_theorem}[{\cite[Theorem 5.2]{RW_SODA}}]
	\label{thm:hg}
	For a monadically dependent graph class $\mathscr C$,
	the following conditions are equivalent:
\begin{enumerate}
	\item 	$\mathscr C$ has a stable edge relation;
	\item $\mathscr  C$ is stable;
	\item $\mathscr C$ is monadically stable
\end{enumerate}
\end{ext_theorem}

\begin{theorem}[Restatement of Theorem~\ref{thm:stable_vm}]
	\label{thm:xstable_vm}
	Let $\mathscr C$ be a hereditary class of graphs.
	Then, $\mathscr C$ is stable if and only if, for every integer $r$, the class $\mathscr C$ excludes some half-graph as a depth-$r$ shallow vertex minor.
\end{theorem}
\begin{proof}
	Assume that for some integer $r$, the class $\svm_r(\mathscr C)$ contains all the half-graphs. Then,  $\svm_r(\mathscr C)$ is unstable and, according to Lemma~\ref{lem:svm_st}, so is $\mathscr C$.
	
	Conversely, assume that $\mathscr C$ is unstable.	
	Assume that $\mathscr C$ is also independent.
	Then, either $\mathscr C$ includes all permutation graphs (including all half-graphs, as they are permutation graphs), or there exists a non-negative $r$ such that $\svm_r(\mathscr C)$ includes all split interval graphs. 
	As every half-graph can be obtained as a depth-$1$ vertex minor  of a split interval graph, we conclude that 
	$\svm_{r+1}(\mathscr C)$ includes all half-graphs.
	
	Otherwise, $\mathscr C$ is dependent and unstable.
	Then, according to Theorem~\ref{thm:hg}, the graphs in $\mathscr C$ contain arbitrarily large semi-induced half-graphs.
	By a standard Ramsey argument, we deduce that we can find in $\mathscr C$ arbitrarily large flipped half graphs, where the partition used for the flip is the bipartition of the half-graph.
	Let $a_1,\dots,a_n,b_1,\dots,b_n$ be the vertices of the flipped half-graph. By considering the independent set $\{a_n,b_1\}$ and Lemma~\ref{lem:svm_flip}, we deduce that  $\svm_2(\mathscr C)$ contains all half-graphs.
	
\end{proof}

%
%
%
%
%
%

\section{Extension to binary structures}
\label{sec:bin}
In this section, we discuss the generalization of \Cref{thm:NIP_vm,thm:stable_vm} to binary structures.
As an application, we show that this generalization allows a short proof that the boundedness of twin-width is preserved by shallow vertex minors.

A \emph{relational signature} $\sigma$ is a set of relation symbols with arity.
A \emph{binary relational structure} is a relational structure whose signature contains relations with arity at most $2$.

Fix a finite binary relational signature $\sigma=(R_1,\dots,R_p, P_1,\dots,P_q)$ (with $R_1,\dots,R_p$ binary and $P_1,\dots,P_q$ unary) and let $\sigma'=(E, Q_1,\dots,Q_p, P_1,\dots,P_q)$ (with $E$ binary and  $Q_1,\dots,Q_p$, $P_1,\dots,P_q$ unary).

Let $\sigma$ be a relational signature and  
let $\bar\sigma=\sigma\cup\{\sim\}$,  where $\sim$ is a binary relation symbol.
The \emph{$k$-copy operation} $\mathsf C_k$ maps a $\sigma$-structure $\mathbf M$ into the $\bar\sigma$-structure
$\mathsf C_k(\mathbf M)$ consisting of $k$ copies of $\mathbf M$
where the copies of each element of $M$ are made adjacent by $\sim$. The copies of a same element are called \emph{clones}. Note that for $k=1$, $\mathsf C_1$ maps each structure $\mathsf M$ to itself.

For a set ${\mathcal U}$ of unary relations, the \emph{coloring operation} $\Gamma_{\mathcal U}$ maps a structure $\mathbf M$ to the set~$\Gamma_{\mathcal U}(\mathbf M)$ of all its ${\mathcal U}$-expansions.

Let $\sigma^+,\sigma'$ be relational structures, where
$\sigma^+\setminus\bar\sigma=\mathcal U$.
A \emph{simple  interpretation} $\mathsf I$ of $\sigma'$-structures in $\sigma^+$-structures is defined by a formula $\nu(x)$ and, for each $R\in\sigma'$ with arity $r$, a
 formula  $\rho_R(\bar x)$ with $|\bar x|=r$ (in
the first-order language of $\sigma^+$-structures).
If $\mathbf M^+$ is a $\sigma^+$-structure, then 
$\mathbf N=\mathsf I(\mathbf M^+)$ is the $\sigma'$-structure with domain $N=\nu(\mathbf M^+)$ where, for each $R\in\sigma'$ with arity $r$, we have $R(\mathbf N)=\rho_R(\mathbf M)\cap \nu(\mathbf M)^r$.
For a set $\Gamma$ of $\sigma^+$-structures we let $\mathsf{I}(\Gamma)=\bigcup_{\mathbf M^+\in \Gamma}\mathsf{I}(\mathbf M^+)$. 

A \emph{transduction} $\mathsf T$ is the composition
$\mathsf I\circ\Gamma_{\mathcal U}\circ \mathsf C_k$ of a copy operation $\mathsf C_k$, a coloring operation~$\Gamma_{\mathcal U}$, and a simple interpretation $\mathsf I$.
In other words, for every $\sigma$-structure  $\mathbf M$ we have
$\mathsf T(\mathbf M)=\{\mathsf I(\mathbf M^+): \mathbf M^+\in\Gamma_{\mathcal U}(\mathsf C_k(\mathbf M))\}$. 

Let $\sigma=\{R_1,\dots,R_k\}$ be a binary relational signature.
We consider two transductions.
First, $\mathsf X=\mathsf I_X\circ\Gamma_{\mathcal U}\circ \mathsf C_k$ from $\sigma$-structures to $\mathcal U$-colored digraphs, where $\mathcal U=\{P_1,\dots,P_k\}$ and
$\mathsf I_X$ is the simple interpretation defined by the formulas
\begin{align*}
	\nu(x)&:= \top\\
	\rho_{P_i}(x)&:=P_i(x)\\
	\rho_E(x,y)&:=\biggl(\bigvee_{i\in [k]}\bigl(P_i(x)\wedge P_i(y)\wedge R_i(x,y)\bigr)\biggr) \vee\biggl(\bigvee_{i\neq j\in [k]}\bigl(P_i(x)\wedge P_j(y)\wedge (x\sim y)\bigr)\biggr)
\end{align*}

Second, $\mathsf K$ is the simple interpretation of  $\sigma$-structures in $\mathcal U$-colored digraphs defined by the formulas

\begin{align*}
	\nu(x)&:=P_1(x)\\
	R_1(x,y)&:=E(x,y)\\
	R_i(x,y)&:=\exists x',y'\ \bigl(P_i(x')\wedge P_i(y')\wedge E(x,x')\wedge E(y,y')\wedge E(x',y')\bigr)&\text{($1<i\leq k$)}
\end{align*}

Particularly, let $f_X(\mathbf M)\in \mathsf X(\mathbf M)$ be the graph obtained from the $\mathcal U$-expansion 
such that $P_i$ marks the $i$th clones.
It is clear that for every $\sigma$-structure $\mathbf M$, we have
$\mathbf M=\mathsf K(f_x(\mathbf M))$. (Hence,
$\mathbf M\in \mathsf K\circ\mathsf X(\mathbf M)$.)
 In particular, a class $\mathscr C$ of $\sigma$-structures is monadically dependent (resp. monadically stable) if and only if $f_X(\mathscr C)$  is monadically dependent (resp. monadically stable).

 Note that if $R_i$ is symmetric, so is the adjacency relation between the vertices in the unary relation $P_i$. In such a case, we can consider that subdigraphs induced by vertices in $P_i$ are actually graphs.
Let $\sigma$ be a finite binary structure and let $M$ be a $\sigma$-structure such that $R_1,\dots,R_a$ are symmetric.
For $i\in [a]$, we define $\mathbf M\ast^{R_i} v$ as the $\sigma$-structure obtained by $R_i$-complementing the $R_i$-neighborhood of $v$.
A \emph{depth-$1$ vertex minor} of $\mathbf M$ has the form
\[
\mathbf M\ast^{R_1} I_1\ast\dots\ast^{R_a} I_a-D,
\]
where $I_i$ is an $R_i$-independent subset of $M$ and $D\subseteq M$. We denote by $\svm_1(\mathbf M)$ the set of all the depth-$1$ vertex minors of $\mathbf M$ and, for a class $\mathscr C$ of $\sigma$-structures, we define
$\svm_1(\mathscr C)=\bigcup_{\mathbf M\in\mathscr C}\svm_1(\mathbf M)$.
\begin{fact}
	For every class $\mathscr C$ of $\sigma$-structures, we have
	\[
	f_X(\svm_1(\mathscr C))\subseteq \svm_1(f_X(\mathscr C)).
	\]
\end{fact}
As a consequence, we have
\begin{corollary}
	\label{cor:binary}
	Let $\mathscr C$ be a hereditary class of binary structures. Then, 
\begin{enumerate}
	\item $\mathscr C$ is stable if and only if $\svm_1(\mathscr C)$ is stable;
	\item $\mathscr C$ is dependent if and only if $\svm_1(\mathscr C)$ is dependent.
\end{enumerate}
\end{corollary}
As an application of this corollary, we have
\begin{theorem}
	Let $\mathscr C$ be a class of graphs. Then, $\mathscr C$ has bounded twin-width if and only if $\svm_1(\mathscr C)$ has bounded twin-width.
\end{theorem}
\begin{proof}
	As $\mathscr C\subseteq \svm_1(\mathscr C)$, the class $\mathscr C$ has bounded twin-width if $\svm_1(\mathscr C)$ has bounded twin-width.
	
	Conversely, assume that $\mathscr C$ has bounded twin-width.
	According to \cite{Tww_ordered}, the class $\mathscr C$ has an expansion to a dependent class $\mathscr C^<$ of ordered graph. This expansion is a binary structure with (binary) signature $\{E,<\}$.
	According to Corollary~\ref{cor:binary}, $\svm_1(\mathscr C^<)$ is dependent. Here, we have a single symmetric relation, which is $E$. As the linear order is not modified by local complementations on $E$-neighborhood, the class $\svm_1(\mathscr C^<)$  is a class $\mathscr D^<$ of ordered graphs, which is a dependent expansion of $\svm_1(\mathscr C)$. According to \cite{Tww_ordered}, it follows that $\svm_1(\mathscr C)$ has bounded twin-width.
\end{proof}

%
%
%

\section{Discussion}


It has been proved \cite{KWON202176} that a class $\mathscr C$  has bounded shrubdepth if and only if $\mathscr C$
 excludes some path as a vertex minor. It is natural to ask whether this result could be strengthened by restricting  to shallow vertex minors.
\begin{problem}
	Is it true that for every class $\mathscr C$ with unbounded shrubdepth there exists an integer $r$ such that $\svm_r(\mathscr C)$ contains all paths or all half-graphs?
\end{problem}
Remark that this problem can be restated as follows: Is it true that for every stable class $\mathscr C$ with unbounded shrubdepth there exists an integer $r$ such that $\svm_r(\mathscr C)$ contains all paths or all half-graphs?
\medskip

It is known \cite{modulo} that a class $\mathscr C$ has structurally bounded expansion if and only if there exists a class $\mathscr D$ of bipartite graphs with bounded expansion such that $\mathscr C\subseteq \svm_1(\mathscr D)$. An obvious question is whether such a kind of characterizations would extend to stable hereditary classes of graphs.
\begin{problem}
	Is it true that  a  hereditary class $\mathscr C$ is stable if and only if there exist an integer $c$ and a nowhere dense class $\mathscr D$, such that $\mathscr C\subseteq \svm_c(\mathscr D)$?
\end{problem}
One direction follows from Theorem~\ref{thm:stable_vm}: If there exist an integer $c$ and a nowhere dense class $\mathscr D$ such that $\mathscr C\subseteq \svm_c(\mathscr D)$, then $\mathscr C$ is stable. On the other hand, it might well  follow from \cite{covers} that if $\mathscr C$ is stable then there exists a stable almost nowhere dense class $\mathscr D$  of bipartite graphs with $\mathscr C\subseteq\svm_1(\mathscr D)$.


\section*{Acknowledgments}
While writing this article, we have been informed that the methods used by Dreier, M\"ahlmann, and Toru\'nczyk might allow to directly prove the preservation of monadic dependence and monadic stability under any transduction based on first-order logic with modulo counting. However, we decided to keep the proofs of the preservation of shallow vertex minors as an illustrative application of the commuting properties of flip and local complementation.
\bibliographystyle{amsplain}
\bibliography{ref}

\providecommand{\bysame}{\leavevmode\hbox to3em{\hrulefill}\thinspace}
\providecommand{\MR}{\relax\ifhmode\unskip\space\fi MR }
\providecommand{\MRhref}[2]{%
  \href{http://www.ams.org/mathscinet-getitem?mr=#1}{#2}
}
\providecommand{\href}[2]{#2}
\begin{thebibliography}{10}

\bibitem{Adler2013}
H.~Adler and I.~Adler, \emph{Interpreting nowhere dense graph classes as a
  classical notion of model theory}, European Journal of Combinatorics
  \textbf{36} (2014), 322--330.

\bibitem{Tww_ordered}
E.~Bonnet, U.~Giocanti, P.~Ossona~de Mendez, P.~Simon, S.~Thomass{\'e}, and
  S.~Toru\'nczyk, \emph{Twin-width {IV}: ordered graphs and matrices}, STOC
  2022: Proceedings of the 54th Annual ACM SIGACT Symposium on Theory of
  Computing, 2022.

\bibitem{twin-width1}
E.~Bonnet, E.J. Kim, S.~Thomass{\' e}, and R.~Watrigant, \emph{Twin-width {I:}
  tractable {FO} model checking}, 61st Annual Symposium on Foundations of
  Computer Science (FOCS 2020), IEEE, 2020, pp.~601--612.

\bibitem{lmcs_perm}
E.~Bonnet, J.~Ne\v{s}et\v{r}il, P.~Ossona~de Mendez, S.~Siebertz, and
  S.~Thomass{\'e}, \emph{Twin-width and permutations}, Logical Methods in
  Computer Science (2024), to appear.

\bibitem{braunfeld2022existential}
S.~Braunfeld and M.~C. Laskowski, \emph{Existential characterizations of
  monadic {NIP}}, arXiv preprint arXiv:2209.05120, 2022.

\bibitem{covers}
S.~Braunfeld, J.~Ne{\v s}et{\v r}il, P.~Ossona~de Mendez, and S.~Siebertz,
  \emph{Decomposition horizons: from graph sparsity to model-theoretic dividing
  lines}, European Journal of Combinatorics (2024), Eurocomb 2023 special
  issue; submitted.

\bibitem{MCST}
J.~Dreier, I.~Eleftheriadis, N.~M{\"a}hlmann, R.~McCarty, M.~Pilipczuk, and
  S.~Toru{\'n}czyk, \emph{First-order model checking on monadically stable
  graph classes}, arXiv preprint arXiv:2311.18740[cs.LO], 2023.

\bibitem{dreier2023indiscernibles}
J.~Dreier, N.~M{\"a}hlmann, S.~Siebertz, and S.~Toru{\'n}czyk,
  \emph{Indiscernibles and flatness in monadically stable and monadically {NIP}
  classes}, 50th International Colloquium on Automata, Languages, and
  Programming (ICALP 2023), Schloss-Dagstuhl-Leibniz Zentrum f{\"u}r
  Informatik, 2023.

\bibitem{dreier2024flipbreakability}
J.~Dreier, N.~M\"ahlmann, and S.~Toru\'nczyk, \emph{Flip-breakability: A
  combinatorial dichotomy for monadically dependent graph classes}, arXiv
  preprint arXiv:2403.15201v1 [math.CO], 2024.

\bibitem{foldes1977split}
S.~Foldes and P.L. Hammer, \emph{Split graphs having {D}ilworth number two},
  Canadian Journal of Mathematics \textbf{29} (1977), no.~3, 666--672.

\bibitem{GEELEN202393}
J.~Geelen, O.~Kwon, R.~McCarty, and P.~Wollan, \emph{The grid theorem for
  vertex-minors}, Journal of Combinatorial Theory, Series B \textbf{158}
  (2023), 93--116, Robin Thomas 1962-2020.

\bibitem{KWON202176}
O.~Kwon, R.~McCarty, S.~Oum, and P.~Wollan, \emph{Obstructions for bounded
  shrub-depth and rank-depth}, Journal of Combinatorial Theory, Series B
  \textbf{149} (2021), 76--91.

\bibitem{Sparsity}
J.~Ne{\v s}et{\v r}il and P.~{Ossona de Mendez}, \emph{Sparsity (graphs,
  structures, and algorithms)}, Algorithms and Combinatorics, vol.~28,
  Springer, 2012, 465 pages.

\bibitem{RW_SODA}
J.~Ne{\v s}et{\v r}il, P.~{Ossona de Mendez}, M.~Pilipczuk, R.~Rabinovich, and
  S.~Siebertz, \emph{Rankwidth meets stability}, Proceedings of the 2021
  ACM-SIAM Symposium on Discrete Algorithms (SODA), 2021, pp.~2014--2033.

\bibitem{modulo}
J.~Ne{\v s}et{\v r}il, P.~Ossona~de Mendez, and S.~Siebertz,
  \emph{Modulo-counting first-order logic on bounded expansion classes},
  Discrete Mathematics (2023), 113700, in press.

\bibitem{OUM200579}
S.~Oum, \emph{Rank-width and vertex-minors}, Journal of Combinatorial Theory,
  Series B \textbf{95} (2005), no.~1, 79 -- 100.

\end{thebibliography}
\end{document}